\documentclass[11pt,a4paper]{amsart}

\usepackage{fancyhdr,euler,stmaryrd} 
\usepackage{amssymb} 
\usepackage{combelow} 
\usepackage{tensor} 
\usepackage{hyperref} 
\usepackage[all]{xy} 
\usepackage{mathrsfs} 
\usepackage{indentfirst} 
\usepackage{mathtools} 
\usepackage{titlesec} 

\newtheorem{theorem}{Theorem}
\newtheorem{corollary}{Corollary}
\newtheorem{definition}{Definition}
\newtheorem{example}{Example}
\newtheorem{lemma}{Lemma}
\newtheorem{proposition}{Proposition}
\newtheorem{remark}{Remark}
\renewenvironment{proof}{\par\noindent\textbf{Proof.}}{\hfill $\blacksquare$\par\smallskip}

\newcommand{\G}{G}
\newcommand{\g}{g}
\newcommand{\h}{h}
\newcommand{\C}{\mathscr{C}}
\renewcommand{\O}{\mathcal{O}}
\newcommand{\Z}{\mathscr{Z}}
\newcommand{\set}[1]{\left\{#1\right\}}
\newcommand{\coref}[1]{Section \S\ref{#1},}
\newcommand{\Hom}{\mathrm{Hom}}
\newcommand{\End}{\mathrm{End}}
\newcommand{\Gr}{\mathrm{Gr}}
\newcommand{\op}{\mathrm{op}}

\newcommand{\titlename}{Graded Morita theory over a $G$-graded $G$-acted algebra}

\newcommand{\authorname}      {Virgilius-Aurelian Minu\cb{t}\u{a}}
\newcommand{\shortauthorname} {V. A. Minu\cb{t}\u{a}}
\newcommand{\universityname}  {Babe\cb{s}-Bolyai University of Cluj-Napoca}
\newcommand{\facultyname}     {Faculty of Mathematics and Computer Science}
\newcommand{\departmentname}  {Department of Mathematics}
\renewcommand{\email}         {minuta.aurelian@math.ubbcluj.ro}

\newcommand{\articleabstract}{We develop a group graded Morita theory over a $G$-graded $G$-acted algebra, where $G$ is a finite group.}
\newcommand{\msc}{16W50, 16D20, 16D90, 16S35.}
\renewcommand{\keywords}{Crossed product, group graded Morita equivalence, centralizer subalgebra, graded rings and modules.}

\titleformat{\section}{\Large\bfseries}{\thesection}{1em}{}
\fancypagestyle{mypagestyle}{
  \fancyhf{}
  \fancyhead[OC]{\textit{\titlename}}
  \fancyhead[EC]{\textit{\shortauthorname}}
  \fancyfoot[C]{\medskip $\leftslice\,$\thepage$\,\rightslice$ }
  
}
\fancypagestyle{firstpagestyle}{
  \fancyhf{}
  \fancyfoot[C]{\medskip $\leftslice\,$\thepage$\,\rightslice$ }
  
}
\title[\titlename]{\LARGE{\titlename}}
\newcommand{\institution}{
\universityname\\
\facultyname\\
\departmentname}
\author[\shortauthorname]{\Large{\authorname}
\medskip\\
{\footnotesize \institution\\
email: \texttt{\email}
}}
\begin{document}
\begin{abstract}
\articleabstract\\[0.1cm]
\textsc{MSC 2010.} \msc\\[0.1cm]
\end{abstract}
\begingroup
\def\uppercasenonmath#1{} 
\let\MakeUppercase\relax 
\maketitle
\endgroup
\thispagestyle{firstpagestyle}

\section{Introduction} \label{s:Introduction}
Let $G$ be a finite group. In this article we develop a group graded Morita theory over a $G$-graded $G$-acted algebra, which is motivated by the problem to give a group graded Morita equivalences version of relations between character triples (see \cite{ch:Spath2018}) as in \cite{MM2}. We will follow, in the development of graded Morita theory over a $G$-graded $G$-acted algebra, the treatment of Morita theory given by C. Faith in 1973 in \cite{book:Faith1973}. Significant in this article is the already developed graded Morita theory. Graded Morita theory started in 1980 when R. Gordon and E. L. Green have characterized graded equivalences in the case of $G=\mathbb{Z}$, in  \cite{article:GG1980}. Furthermore, in 1988 it was observed to work for arbitrary groups $G$ by C. Menini and C. N\u{a}st\u{a}sescu, in \cite{article:MN1988}. We will make use of their results under the form given by A. del R\'io in 1991 in \cite{article:delRio1991} and we will also use the graded Morita theory developed by P. Boisen in 1994 in \cite{article:Boisen1994}.
\medskip
\par This paper is organized as follows: In \coref{s:Notations} we introduce the general notations. In \coref{s:graded_bimodules_over_C} we recall from \cite{MM2} the definition of a $G$-graded $G$-acted algebra, we fix one to which we will further refer to by $\C$, we recall the definition of a $G$-graded algebra over said $\C$, and we recall the definition of a graded bimodule over $\C$. Moreover, we will give some new examples, useful for this article, for each notion. In \coref{s:graded_equivalence_data_over_C} we construct the notion of a $G$-graded Morita context over $\C$ and we will give an appropriate example. In \coref{s:Morita_theorems} we introduce the notions of graded functors over $\C$ and of graded Morita equivalences over $\C$ and finally we state and prove two Morita-type theorems using the said notions.

\section{Notations and preliminaries} \label{s:Notations}
Throughout this article, we will consider a finite group $G$. We shall denote its identity by $1$.
\medskip
\par All rings in this paper are associative with identity $1\neq 0$ and all modules are unital and finitely generated. We consider $\O$ to be a commutative ring.
\medskip
\par Let $A$ be a ring. We denote by $A\text{-Mod}$ the category of all left $A$-modules. We shall usually write actions on the left, so in particular, by module we will usually mean a left module, unless otherwise stated.  The notation $\tensor*[_A]{M}{}$ (respectively, $\tensor*[_A]{M}{_{A'}}$) will be used to emphasize that $M$ is a left $A$-module (respectively, an $(A,A')$-bimodule).
\medskip
\par
Let $A=\bigoplus_{\g\in\G} A_{\g}$ be a $\G$-graded $\O$-algebra. We denote by $A\text{-}\Gr$ the cate\-gory of all $\G$-graded left $A$-modules. The forgetful functor from $A\text{-}\Gr$ to $A\text{-Mod}$ will be denoted by $U$. For $M=\bigoplus_{\g\in\G} M_{\g}\in A\text{-}\Gr$ and $\g\in \G$, the $\g$-suspension of $M$ is defined to be the $\G$-graded $A$-module $M(\g)=\oplus_{\h\in\G}M(\g)_{\h}$, where $M(g)_{\h}=M_{\g\h}$. For any $g\in G$, $T_g^A : A\text{-}\Gr\to A\text{-}\Gr$ will denote (as in \cite{article:delRio1991}) the $g$-suspension functor, i.e. $T_g^A(M)=M(g)$ for all $g\in G$. The stabilizer of $M$ in $G$ is, by definition \cite[\S2.2.1]{book:Marcus1999}, the subgroup
\[G_M=\set{g\in G\mid M\simeq M(g)\text{ as }G\text{-graded left }A\text{-modules}}.\]
Let $M,N\in A\text{-}\Gr$. We de\-note by $\Hom_A(M,N),$ the additive group of all $A$-linear homomorphisms from $M$ to $N$. Because $G$ is finite, E. C. Dade showed in \cite[Corollary 3.10]{article:Dade1980} that $\Hom_A(M,N)$ is $G$-graded. More precisely, if $g\in G$, the component of degree $g$ (furthermore called the $g$-component) is defined as in \cite[1.2]{book:Marcus1999}:
\[\Hom_A(M,N)_g:=\set{f\in \Hom_A(M,N)\mid f(M_x)\subseteq N_{xg},\text{ for all }x\in G}.\]
\par We denote by $\text{id}_X$ the identity map defined on a set $X$.

\section{Graded bimodules over a \texorpdfstring{$G$}{G}-graded \texorpdfstring{$G$}{G}-acted algebra}\label{s:graded_bimodules_over_C}

\par We consider the notations given in Section \S\ref{s:Notations}. We recall the definition of a $\G$-graded $\G$-acted algebra and an example of a $\G$-graded $\G$-acted algebra as in \cite{MM2}:
\begin{definition}
 An algebra $\C$ is a $\G$-graded $\G$-acted algebra if
 \begin{enumerate}
	\item $\C$ is $\G$-graded, i.e. $\C=\oplus_{\g\in \G}\C_{\g}$;
	\item $\G$ acts on $\C$ (always on the left in this article);
	\item $\forall \h\in \G,\,\forall c\in \C_{\h}$ we have that $\tensor*[_{}^{\g}]{c}{_{}^{}}\in \C_{\tensor*[_{}^{\g}]{\h}{_{}^{}}}$ for all $\g\in \G$.
 \end{enumerate}
 We denote the identity component (the 1-component) of $\C$ by $\Z:=\C_1$, which is a $\G$-acted algebra.
\end{definition}

\par  Let $A=\bigoplus_{\g\in\G} A_{\g}$ be a strongly $G$-graded $\O$-algebra with identity component $B:=A_1$. For the sake of simplicity, we assume that $A$ is a crossed product (the generalization is not difficult, see for instance \cite[\S 1.4.B.]{book:Marcus1999}). This means that we can choose invertible homogeneous elements $u_{\g}$ in the component $A_{\g}$.

\begin{example} \label{ex:Miyashita}
By Miyashita's theorem \cite[p.22]{book:Marcus1999}, we know that the centralizer $C_A(B)$ is a $\G$-graded $\G$-acted $\O$-algebra: for all $\h\in\G$ we have that $$C_A(B)_{\h}=\set{a\in A_{\h}\mid ab=ba,\,\forall b\in B},$$ and the action is given by $\tensor*[^{\g}]{c}{}=u_{\g}cu^{-1}_{\g},\,\forall \g\in \G,\,\forall c\in C_A(B)$. Note that this definition does not depend on the choice of the elements $u_{\g}$ and that $C_A(B)_1=Z(B)$ (the center of $B$).
\end{example}

\par We recall the definition of a $\G$-graded $\O$-algebra over a $\G$-graded $\G$-acted algebra $\C$ as in \cite{MM2}:

\begin{definition}
		 Let $\C$ be a $\G$-graded $\G$-acted $\O$-algebra. We say that $A$ is a $\G$-graded $\O$-algebra over $\C$ if there is a $\G$-graded $\G$-acted algebra homomorphism
		 $$\zeta:\C\to C_A(B),$$
		 i.e. for any $\h\in \G$ and $c\in \C_h$, we have $\zeta(c)\in C_A(B)_{\h}$, and for every $\g\in\G$, we have $\zeta(\tensor*[^{\g}]{c}{})=\tensor*[^{\g}]{\zeta(c)}{}$.
\end{definition}

\par  An important example of a $\G$-graded $\O$-algebra over a $\G$-graded $\G$-acted algebra, is given by the following lemma:

\begin{lemma} \label{lemma:zetaprimeq}
Let $P$ be a $\G$-graded $A$-module. Assume that $P$ is $\G$-invariant. Let $A'=\mathrm{End}_A(P)^{\text{op}}$ be the set of all $A$-linear endomorphisms of $P$. Then $A'$ is a $\G$-graded $\O$-algebra over $C_A(B)$.
\end{lemma}
\begin{proof}
By \cite[Theorem 2.8]{article:Dade1980}, we have that there exists some $U\in B-$mod such that $P$ and $A\otimes_B U$ are isomorphic as $G$-graded left $A$-modules, henceforth for simplicity we will identify $P$ as $A\otimes_B U$.
\par Because $G$ is finite, E. C. Dade proved in \cite[Corollary 3.10 and \S4]{article:Dade1980} that $A'=\mathrm{End}_A(P)^{\text{op}}$ is a $\G$-graded $\O$-algebra. Moreover, by \cite[\S 4]{article:Dade1980} we have that $P$ is actually a $\G$-graded $(A,A')$-bimodule.
\par Now, the assumption that $P$ is $\G$-invariant, according to \cite[Corollary 5.14]{article:Dade1980} and \cite[\S 2.2.1]{book:Marcus1999}, implies that $A'=\mathrm{End}_A(P)^{\text{op}}$ is a crossed product and that $P$ is isomorphic to its $g$-suspension, $P(\g)$, for all $\g\in\G$. Henceforth, we can choose invertible homogeneous elements $u'_{\g}$ in the component $A'_{\g}$, for all $\g\in \G$ such that
$$u'_{\g}:P\to P(\g).$$
\par By taking the truncation functor $(-)_1$ (more details are given in  \cite{article:Dade1980}) we obtain the isomorphism:
$$(u'_{\g})_1:P_1\to (P(\g))_1,$$
where $P_1=B\otimes_B U\simeq U$ and $(P(\g))_1=A_{\g}\otimes_B U=u_{\g}B\otimes_B U$.
We fix arbitrary $a\in A$ and $u\in U$. We have:
$$u'_{\g}(a\otimes_B u)=au'_{\g}(1_A\otimes_B u),$$
but $1_A\otimes_B u\in P_1$, henceforth:
$$u'_{\g}(a\otimes_B u)=a(u'_{\g})_1(1_A\otimes_B u),$$
but there exists an unique $b\in B$ such that $(u'_{\g})_1(1_A\otimes_B u)=u_{\g}b\otimes_B u=u_{\g}\otimes_B bu$. Therefore, by defining $\varphi_{\g}(u):=bu$, we obtain a map $\varphi_{\g}:U\to U$, which is clearly well-defined. Moreover, we have:
$$u'_{\g}(a\otimes_B u)=au_{\g}\otimes_B \varphi_{\g}(u),\text{ for all }a\in A\text{ and }u\in U.$$
\par It is straightforward to prove that $\varphi_{\g}:U\to U$ admits an inverse and that
$$u'^{-1}_{\g}(a\otimes_B u)=au^{-1}_{\g}\otimes_B \varphi^{-1}_{\g}(u),\text{ for all }a\in A\text{ and }u\in U.$$
\par We consider the $\G$-graded algebra homomorphism from \cite[Lemma 3.2.]{MM1}:
\[\theta:C_A(B)\to A'=\mathrm{End}_A(P)^{\text{op}},\qquad 	\theta(c)(a\otimes u)=ac\otimes u,\]
 where $c\in C_A(B)$, $a\in A$, and $u\in U$. First, we will prove that the image of $\theta$ is a subset of $C_{A'}(B')$. Indeed, consider $b'\in B'$ and $c\in C_A(B)$. We want:
$$\theta(c)\circ b'=b'\circ \theta(c).$$
 Consider $a\otimes u\in A\otimes U=P$. We have:
 \[
	(\theta(c)\circ b')(a\otimes u)  =  \theta(c)(b'(a\otimes u)) =  a\theta(c)(b'(1_A\otimes u)),
 \]
 because $b'$ and $\theta(c)$ are $A$-linear. We fix $b'(1_A\otimes u)=a_0\otimes u_0\in A\otimes_B U$, but because $b'\in B'=A'_1=\mathrm{End}_A(P)^{\text{op}}_1$ we know that $b'$ preserves the grading, so $1_A\in A_1$ implies that $a_o\in B$. Hence:
  \[
	(\theta(c)\circ b')(a\otimes u)  =  a\theta(c)(a_0\otimes u_0) =  aa_0c\otimes u_0=  aca_0\otimes u_0.
\]
 Following, we have that:
   $$
\begin{array}{rcl}
	(b'\circ \theta(c))(a\otimes u) & = & (b'( \theta(c)(a\otimes u)) =  b'(ac\otimes u) \\
	& = & acb'(1_A\otimes u) =  aca_0\otimes u_0.
\end{array} $$
 Henceforth, the image of $\theta$ is a subset of $C_{A'}(B')$. Second, we prove that $\theta$ is $\G$-acted, in the sense that:
 $$\theta(\tensor[^{\g}]{c}{})=\tensor[^{\g}]{(\theta(c))}{},\text{ for all }\g\in\G.$$
 Indeed, we fix $\g\in\G$, and $a\otimes u\in A\otimes_B U=P$. We have:
    \[
	\theta(\tensor[^{\g}]{c}{})(a\otimes u)   = \theta(u_{\g}cu^{-1}_{\g})(a\otimes u)=  au_{\g}cu^{-1}_{\g}\otimes u
\]
 and
 $$
\begin{array}{rcl}
	\tensor[^{\g}]{(\theta(c))}{}(a\otimes u) & = & (u'_{\g}\cdot \theta(c)\cdot u'^{-1}_{\g})(a\otimes u)= (u'^{-1}_{\g}\circ \theta(c)\circ u'_{\g})(a\otimes u)\\
	& = & u'^{-1}_{\g}(\theta(c)(u'_{\g}(a\otimes u))) =  u'^{-1}_{\g}(\theta(c)(au_{\g}\otimes \varphi_{\g}(u)))\\
	& = & u'^{-1}_{\g}(au_{\g}c\otimes \varphi_{\g}(u)) =  au_{\g}cu^{-1}_{\g}\otimes \varphi^{-1}_{\g}(\varphi_{\g}(u))\\
	& = & au_{\g}cu^{-1}_{\g}\otimes u.
\end{array}
 $$
\par Finally, by taking $\zeta':C_A(B)\to C_{A'}(B')$ to be the corestriction of $\theta$ to $C_{A'}(B')$, we obtain that $A'$ is a $\G$-graded $\O$-algebra over $C_A(B)$, via the $\G$-graded $\G$-acted homomorphism $\zeta'$.
\end{proof}

\par Let $A'=\bigoplus_{\g\in\G} A'_{\g}$ be another strongly $\G$-graded $\O$-algebra with the identity component $B':=A'_1$. Again, we will consider that also $A'$ is a crossed product, hence we will choose invertible homogeneous elements $u'_{\g}$ in the component $A'_{\g}$, for all $\g\in \G$.

\par  Now, we assume that $A$ and $A'$ are both strongly $\G$-graded $\O$-algebras over a $\G$-graded $\G$-acted algebra $\C$, each endowed with a $\G$-graded $\G$-acted algebra homomorphism $\zeta:\C\to C_A(B)$ and $\zeta':\C\to C_{A'}(B')$ respectively.
\par We recall the definition of a $\G$-graded bimodule over $\C$ as in \cite{MM2}:
\begin{samepage}
\begin{definition} \label{def:graded_bimodules_over_c}
	We say that $\tilde{M}$ is a $\G$-graded $(A,A')$-bimodule over $\C$ if:
	\begin{enumerate}
		\item $\tilde{M}$ is an $(A,A')$-bimodule;
		\item $\tilde{M}$ has a decomposition $\tilde{M}=\bigoplus_{\g\in\G}\tilde{M}_{\g}$ such that $A_{\g}\tilde{M}_{x}A'_{\h}\subseteq \tilde{M}_{\g x\h}$, for all $\g,x,\h\in \G$;
		\item $\tilde{m}_{\g}\cdot c=\tensor*[^{\g}]{c}{}\cdot \tilde{m}_{\g}$, for all $c\in \C,\tilde{m}_{\g}\in\tilde{M}_{\g},\g\in \G$, where $c \cdot \tilde {m} = \zeta(c)\cdot \tilde{m}$ and $\tilde{m}\cdot c=\tilde{m}\cdot \zeta'(c)$, for all $c\in \C,\tilde{m}\in\tilde{M}$.
	\end{enumerate}
\end{definition}
\end{samepage}
\begin{remark}
	Condition (3) of Definition \ref{def:graded_bimodules_over_c} can be rewritten (see \cite{MM2} for the proof) as follows:
	\begin{enumerate}
		\setcounter{enumi}{2}
		\item[(3')] $m\cdot c= c\cdot m$, for all $c\in \C$, $m\in \tilde{M}_1$.
	\end{enumerate}
\end{remark}

\par  An example of a $\G$-graded bimodule over a $\G$-graded $\G$-acted algebra is given by the following proposition:
\begin{proposition} \label{prop:P_bimod_over_C}
Let $\C$ be a $\G$-graded $\G$-acted algebra and $A$ a strongly $\G$-graded $\O$-algebra over $\C$. Let $P$ be a $\G$-invariant $\G$-graded $A$-module. Let $A'=\mathrm{End}_A(P)^{\text{op}}$. Then the following statements hold:
\begin{enumerate}
	\item $A'$ is a $\G$-graded $\O$-algebra over $\C$;
	\item $P$ is a $\G$-graded $(A,A')$-bimodule over $\C$.
\end{enumerate}
\end{proposition}
\begin{proof}
\textit{(1)} By Lemma \ref{lemma:zetaprimeq}, we know that $A'$ is $\G$-graded $\O$-algebra over $C_A(B)$ and let $\theta:C_A(B)\to C_{A'}(B')$ be its $G$-graded $G$-acted structure homomorphism. Now, given that $A$ is a strongly $\G$-graded $\O$-algebra over $\C$, we have a $\G$-graded $\G$-acted algebra homomorphism $\zeta:\C\to C_A(B)$ and by taking $\zeta':\C\to C_{A'}(B')$ to be the $\G$-graded $\G$-acted algebra homomorphism obtained by composing $\zeta$ with $\theta$, we obtain that $A'$ is also a $\G$-graded $\O$-algebra over $\C$, with its structure given by $\zeta'$. Hence, the first statement of this proposition was proved.
\par \textit{(2)} Without any loss in generality, we will identify $P$ with $A\otimes_B U$, for some $U\in B-$mod. Following the proof of Lemma \ref{lemma:zetaprimeq} we know that $P$ is a $\G$-graded $(A,A')$-bimodule. We now check that $P$ is $\G$-graded $(A,A')$-bimodule over $\C$.
Indeed, we fix $\g\in\G$, $p_{\g}=a_{\g}\otimes u\in P_{\g}$ and $c\in \C$. We have:
\[
p_{\g}\cdot c    =   (a_{\g}\otimes u) \cdot c  = (a_{\g}\otimes u) \cdot \zeta'(c)   =  a_{\g}\zeta(c)\otimes u,
\]
but $a_{\g}\in A_{\g}$ so there exists a $b\in B$ such that $a_{\g}=u_{\g}b$, therefore:
$$
\begin{array}{rcl}
p_{\g}\cdot c  & = & u_{\g}b\zeta(c)\otimes u =  u_{\g}\zeta(c)b\otimes u\\
& = & u_{\g}\zeta(c)u^{-1}_{\g}u_{\g}b\otimes u =  \tensor[^{\g}]{\zeta(c)}{}a_{\g}\otimes u\\
& = & \zeta(\tensor[^{\g}]{c)}{}a_{\g}\otimes u =  \zeta(\tensor[^{\g}]{c}{})(a_{\g}\otimes u)\\
& = & \zeta(\tensor[^{\g}]{c}{})\cdot p_{\g} =  \tensor[^{\g}]{c}{}\cdot p_{\g}.\\
\end{array}
$$
Henceforth, the last statement of this proposition has been proved.
\end{proof}

\section{Graded Morita contexts over \texorpdfstring{$\C$}{C}} \label{s:graded_equivalence_data_over_C}
We consider the notations given in Section \S\ref{s:Notations}. Let $\C$ be a $G$-graded $G$-acted algebra. We introduce the notion of a $\G$-graded Morita context over $\C$, following the treatment given in \cite[\S12]{book:Faith1973}.
\begin{samepage}
\begin{definition}
	Consider the following Morita context:
	$$(A,A',\tilde{M},\tilde{M}',f,g).$$
	We call it a $G$-graded Morita context over $\C$ if:
	\begin{enumerate}
		\item $A$ and $A'$ are strongly $\G$-graded $\O$-algebras over $\C$;
		\item $\tensor*[_{A}]{\tilde{M}}{_{A'}}$ and $\tensor*[_{A'}]{\tilde{M}}{_{A}^{'}}$ are $\G$-graded bimodules over $\C$;
		\item $f:\tilde{M}\otimes_{A'}\tilde{M}'\to A$ and $g:\tilde{M}'\otimes_{A}\tilde{M}\to A'$ are $\G$-graded bimodule homomorphisms such that by setting $f(\tilde{m}\otimes\tilde{m}')=\tilde{m}\tilde{m}'$ and $g(\tilde{m}'\otimes\tilde{m})=\tilde{m}'\tilde{m}$, we have the associative laws:
		\[
	(\tilde{m}\tilde{m}')\tilde{n}=\tilde{m}(\tilde{m}'\tilde{n})\quad\text{and}\quad (\tilde{m}'\tilde{m})\tilde{n}'=\tilde{m}'(\tilde{m}\tilde{n}'),
	\]
	for all $\tilde{m},\tilde{n}\in\tilde{M}$, $\tilde{m}',\tilde{n}'\in\tilde{M}'$.
	\end{enumerate}
	If $f$ and $g$ are isomorphisms, then $(A,A',\tilde{M},\tilde{M}',f,g)$ is called a surjective $G$-graded Morita context over $\C$.
\end{definition}
\end{samepage}

%
\par  As an example of a $\G$-graded Morita context over $\C$, we have the following proposition which arises from \cite[Proposition 12.6]{book:Faith1973}.

\begin{proposition} \label{prop:preequiv_example}
Let $A$ be a strongly $\G$-graded $\O$-algebra over $\C$, let $P$ be a $\G$-invariant $\G$-graded $A$-module, let $A'=\mathrm{End}_A(P)^{\text{op}}$ and let $P^{\ast}:=\Hom_A(P,A)$ be the $A$-dual of $P$. Then
$$(A,A',P,P^{\ast},(\cdot,\cdot),[\cdot,\cdot])$$
is a $\G$-graded Morita context over $\C$, where $(\cdot,\cdot)$ is a $\G$-graded $(A,A)$-homo\-morphism, called the evaluation map, defined by:
\[
\begin{array}{l}
(\cdot,\cdot):P\otimes_{A'}P^{\ast}\to A,\\
x\otimes \varphi\mapsto \varphi(x),\text{ for all }\varphi\in P^{\ast},\,x\in P,
\end{array}
\]
and where $[\cdot,\cdot]$ is a $\G$-graded $(A',A')$-homomorphism defined by:
\[
\begin{array}{l}
[\cdot,\cdot]:P^{\ast}\otimes_{A}P\to A',\\
\varphi\otimes x\mapsto  [\varphi,x],\text{ for all }\varphi\in P^{\ast},\,x\in P,
\end{array}
\]
where for every $\varphi\in P^{\ast}$ and $x\in P$, $[\varphi,x]$ is an element of $A'$ such that
$$
y[\varphi,x]=\varphi(y)\cdot x,\text{ for all }y\in P.
$$
\end{proposition}
\begin{proof}
For the sake of simplicity, we will assume that $A$ is a crossed product as in Section \S\ref{s:graded_bimodules_over_C}.
By Proposition \ref{prop:P_bimod_over_C}, we have that $A'$ is also a $\G$-graded $\O$-algebra over $\C$ and that $P$ is a $\G$-graded $(A,A')$-bimodule over $\C$.
Now, it is known that the $A$-dual of $P$, $P^{\ast}:=\Hom_A(P,A)$ is an $(A',A)$-bimodule, where for each $\varphi\in P^{\ast}$ and for each $p\in P$, we have:
$$(a'\varphi a)(p)=(\varphi(pa'))a,$$
for all $a'\in A'$ and $a\in A$. By \cite[\S 1.6.4.]{book:Marcus1999}, we know that $P^{\ast}$ is actually a $\G$-graded $(A',A)$-bimodule, where for all $\g\in\G$, the $\g$-component is defined as follows:
$$P^{\ast}_{\g}=\set{\varphi\in P^{\ast}\mid \varphi(P_{x})\subseteq A_{x\g},\,\text{for all }x\in\G}.$$
We prove that $P^{\ast}$ is a $\G$-graded $(A',A)$-bimodule over $\C$. Consider $\g,\h\in \G$, $\varphi_{\g}\in P^{\ast}_{\g}$, $c\in C$ and $p_{\h}\in P_h$. We have:
\[
	(\varphi_{\g}c)(p_{\h}) = (\varphi_{\g})(p_{\h})c.
\]
Because $(\varphi_{\g})(p_{\h})\in A_{\h\g}$ we can choose a homogeneous element $u_{\h\g}\in A_{\h\g}$ and $b\in B$ such that $(\varphi_{\g})(p_{\h})=u_{\h\g}b$. Henceforth,
$$
\begin{array}{rcl}
	(\varphi_{\g}c)(p_{\h}) & = & u_{\h\g}bc =  u_{\h\g}cb\, = \, u_{\h\g}cu^{-1}_{\h\g}u_{\h\g}b
	\, = \, \tensor*[^{\h\g}]{c}{}u_{\h\g}b\\
	& = & \tensor*[^{\h\g}]{c}{}(\varphi_{\g})(p_{\h}) =  (\varphi_{\g})(\tensor*[^{\h\g}]{c}{}p_{\h}) \\
	& = & (\varphi_{\g})(p_{\h}\tensor*[^{\g}]{c}{})
	\, = \, (\tensor*[^{\g}]{c}{}\varphi_{\g})(p_{\h}),
\end{array}
$$
thus $\varphi_{\g}c = \tensor*[^{\g}]{c}{}\varphi_{\g}$, therefore $P^{\ast}$ is a $\G$-graded $(A',A)$-bimodule over $\C$. Next, following \cite[\S 12]{book:Faith1973}, it is clear that $(\cdot,\cdot)$ and $[\cdot,\cdot]$ are an $(A,A)$-homomorphism and an $(A',A')$-homomorphism, respectively. We now check if they are graded, as in the sense of \cite[\S 3]{article:Boisen1994}: Indeed, consider $p_{\g}\in P_{\g}$ and $\varphi_{\h}\in  P^{\ast}_{\h}$. We have:
\[(p_{\g},\varphi_{\h})=\varphi_{\h}(p_{\g}),\]
which is an element of $A_{\g\h}$, given the gradation of $P^{\ast}:=\Hom_A(P,A)$. Also, for every $y_{k}\in P_{k}$, we have:
\[[\varphi_{\h},p_{\g}](y_{k})=\varphi_h(y_{k})p_{\g},\]
which is an element of $A_{k\h\g}$, given the gradation of $P^{\ast}$ and of $P$, therefore $[\varphi_{\h},p_{\g}]$ is an element of $A_{\h\g}$.
Finally, we verify the associative law of the two homomorphisms: Let $p,q\in P$ and $\varphi,\psi\in P^{\ast}$. We have:
\[
(p,\varphi)q=\varphi(p)q
\quad\text{and}\quad
p[\varphi,q]=\varphi(p)q,
\]
hence
\[
(p,\varphi)q=p[\varphi,q];
\]
and for all $y\in P$, we have:
\[
([\varphi,p]\psi)(y)=\psi(y[\varphi,p])=\psi(\varphi(y)p)=\varphi(y)\psi(p),
\]
because $\varphi(y)\in A$ and $\psi$ is $A$-linear, and also we have
\[
(\varphi(p,\psi))(y)=(\varphi\psi(p))(y)=\varphi(y)\psi(p),
\]
because $\psi(p)\in A$, hence $$[\varphi,p]\psi=\varphi(p,\psi).$$
Therefore $(A,A',P,P^{\ast},(\cdot,\cdot),[\cdot,\cdot])$ is a $\G$-graded Morita context over $\C$.
\end{proof}

\par  If $(A,A',\tilde{M},\tilde{M}',f,g)$ is a surjective $\G$-graded Morita context over $\C$, then by Proposition 12.7 of \cite{book:Faith1973}, we have that $A'$ is isomorphic to $\mathrm{End}_A(\tilde{M})^{\text{op}}$ and we have a bimodule isomorphism between $\tilde{M}'$ and $\tilde{M}^{\ast}=\Hom_A(\tilde{M},A)$. Henceforth, in this situation, the example given by Proposition \ref{prop:preequiv_example} is essentially unique up to an isomorphism.

\par Given Corollary 12.8 of \cite{book:Faith1973}, the example given by Proposition \ref{prop:preequiv_example} is a surjective $\G$-graded Morita context over $\C$ if and only if $\tensor[_{A}]{P}{}$ is a progenerator.

\section{Graded Morita theorems over \texorpdfstring{$\C$}{C}} \label{s:Morita_theorems}

\par  Consider the notations given in \S\ref{s:Notations}. Let $\C$ be a $G$-graded $G$-acted algebra. We denote by $A$ and $A'$ two strongly $G$-graded $\O$-algebras over $\C$ (with identity components $B:=A_1$ and $B':=A'_1$), each endowed with a $\G$-graded $\G$-acted algebra homomorphism $\zeta:\C\to C_A(B)$ and $\zeta':\C\to C_{A'}(B')$ respectively.  According to \cite{article:delRio1991} we have the following definitions:

\newcommand{\F}{\tilde{\mathcal{F}}}
\begin{definition}
\begin{enumerate}
	\item We say that the functor $\F:A\text{-}\Gr\to A'\text{-}\Gr$ is $G$-graded if for every $g\in G$, $\F$ commutes with the $g$-suspension functor, i.e. $\F\circ T^A_g$ is naturally isomorphic to $T^{A'}_g\circ \F$;
	\item We say that $A$ and $A'$ are $G$-graded Morita equivalent if there is a $\G$-graded equivalence: $\F:A\text{-}\Gr\to A'\text{-}\Gr$.
\end{enumerate}
\end{definition}

Assume that $A$ and $A'$ are $\G$-graded Morita equivalent. Therefore, we can consider the $G$-graded functors:
$$
\xymatrix@C+=3cm{
A\text{-}\Gr \ar@<+.5ex>[r]^{\tilde{\mathcal{F}}} & A'\text{-}\Gr \ar@<+.5ex>[l]^{\tilde{\mathcal{G}}}
}
$$
which give a $\G$-graded Morita equivalence between $A$ and $A'$. By Gordon and Green's result \cite[Corollary 10]{article:delRio1991}, this is equivalent to the existence of a Morita equivalence between $A$ and $A'$ given by the following functors:
$$
\xymatrix@C+=3cm{
A\text{-Mod} \ar@<+.5ex>[r]^{\mathcal{F}} & \ar@<+.5ex>[l]^{\mathcal{G}} A'\text{-Mod};
}
$$
such that the following diagram is commutative:
$$
\xymatrix@C+=3cm{
A\text{-}\Gr \ar[d]_{U} \ar@<+.5ex>[r]^{\tilde{\mathcal{F}}} & A'\text{-}\Gr \ar[d]^{U'} \ar@<+.5ex>[l]^{\tilde{\mathcal{G}}}\\
A\text{-Mod} \ar@<+.5ex>[r]^{\mathcal{F}} & \ar@<+.5ex>[l]^{\mathcal{G}} A'\text{-Mod}
}
$$
in the sense that:
\[
U'\circ \tilde{F} =  F\circ U, \ \qquad U\circ \tilde{G}  =  G\circ U',
\]
where $U'$ is the forgetful functor from $A'\text{-}\Gr$ to $A'\text{-Mod}$.
\renewcommand{\P}{\tilde{P}}
\begin{lemma}\label{remark:same_stabilizer}
If $\tilde{P}$ is a $G$-graded $A$-module, then $\tilde{P}$ and $\F(\tilde{P})$ have the same stabilizer in $G$.
\end{lemma}
\begin{proof}
Let $g\in G_{\P}$. We have $\P\simeq \P(g)$ as $G$-graded $A$-modules. Because $\F$ is a graded functor, we have that it commutes with the $g$-suspension functor. Thus $\F(\tilde{P}(g))\simeq \F(\tilde{P})(g)$ in $A'\text{-}\Gr$. Henceforth, $\F(\tilde{P})\simeq \F(\tilde{P})(g)$ in $A'\text{-}\Gr$, thus $g\in G_{\F(\P)}$. Hence $G_{\P}\subseteq G_{\F(\P)}$. The converse, $G_{\F(\P)}\subseteq G_{\P}$, is straightforward, thus $G_{\P}= G_{\F(\P)}$.
\end{proof}
\par  Consider $\tilde{P}$ and $\tilde{Q}$ two $\G$-graded $A$-modules. 
We have the following morphism:
\begin{equation}\label{eqn:morphism}
\xymatrix@C+=3cm{
\Hom_{A}(\tilde{P},\tilde{Q}) \ar[r]^{\mathcal{\tilde{F}}}& \Hom_{A'}(\mathcal{\tilde{F}}(\tilde{P}),\mathcal{\tilde{F}}(\tilde{Q})).
}
\tag{$\ast$}
\end{equation}

\par By following the proofs of Lemma \ref{lemma:zetaprimeq} and Proposition \ref{prop:P_bimod_over_C}, we have a $G$-graded homomorphism from $\C$ to $\mathrm{End}_A(\tilde{P})^{\text{op}}$ (the composition between the structure homomorphism $\zeta:\C\to C_A(B)$ and the morphism $\theta:C_A(B)\to \mathrm{End}_A(\tilde{P})^{\text{op}}$ from \cite[Lemma 3.2.]{MM1}) and that $\tilde{P}$ is a $\G$-graded $(A,\mathrm{End}_A(\tilde{P})^{\text{op}})$-bimodule. Then, by the restriction of scalars we obtain that $\tilde{P}$ is a right $\C$-module. Analogously $\tilde{Q}$, $\mathcal{\tilde{F}}(\tilde{P})$ and $\mathcal{\tilde{F}}(\tilde{Q})$ are also right $\C$-modules, thus $\Hom_{A}(\tilde{P},\tilde{Q})$ and $ \Hom_{A}(\mathcal{\tilde{F}}(\tilde{P}),\mathcal{\tilde{F}}(\tilde{Q}))$ are $\G$-graded $(\C,\C)$-bimodules. This allows us to state the following definition:
\begin{definition}\normalfont
\par
\begin{enumerate}
	\item We say that the functor $\F$ is over $\C$ if the morphism $\F$ (see (\ref{eqn:morphism})) is a morphism of $\G$-graded $(\C,\C)$-bimodules;
	\item We say that $A$ and $A'$ are $G$-graded Morita equivalent over $\C$ if there is a $\G$-graded equivalence over $\C$: $\F:A\text{-}\Gr\to A'\text{-}\Gr$.
\end{enumerate}
\end{definition}

\begin{theorem}[Graded Morita I over $\C$]
	Let $(A,A',\tilde{M},\tilde{M}',f,g)$ be a surjective $\G$-graded Morita context over $\C$. Then the functors
	$$
\begin{array}{c}
	\tilde{M}'\otimes_{A}-:A\text{-}\Gr\to A'\text{-}\Gr\\
	\tilde{M}\otimes_{A'}-:A'\text{-}\Gr\to A\text{-}\Gr
\end{array}	
	$$
	are inverse $G$-graded equivalences over $\C$.
\end{theorem}
\begin{proof}
	Given \cite[Theorem 3.2 (Graded Morita I) 6.]{article:Boisen1994} we already know that the pair of functors $\tilde{M}'\otimes_{A}-$ and $\tilde{M}\otimes_{A'}-$ are inverse $G$-graded equivalences. It remains to prove that they are also over $\C$. We will only prove that the functor $\tilde{M}'\otimes_{A}-$ is over $\C$ because the proof for the latter functor is similar.
	 Consider $\tilde{P}$ and $\tilde{Q}$ two $\G$-graded $A$-modules. 
\par First, we will prove that the morphism
\begin{equation}\label{eqn:morphism_2}
\tilde{M}'\otimes_{A}-: \Hom_A(\tilde{P},\tilde{Q})\to \Hom_{A'}(\tilde{M}'\otimes_{A}\tilde{P},\tilde{M}'\otimes_{A}\tilde{Q})
\tag{$\ast\ast$}
\end{equation}
is a $(\C,\C)$-bimodule homomorphism. Indeed, consider $\varphi\in \Hom_A(\tilde{P},\tilde{Q})$, then $\tilde{M}'\otimes_{A}\varphi\in \Hom_{A'}(\tilde{M}'\otimes_{A}\tilde{P},\tilde{M}'\otimes_{A}\tilde{Q})$. Consider $c,c'\in\C$. We only need to prove that
	 $$\tilde{M}'\otimes_{A}(c\varphi c')=c(\tilde{M}'\otimes_{A}\varphi)c'.$$
	Let $\tilde{m}'\in\tilde{M}'$ and $\tilde{p}\in\tilde{P}$. We have:
	\[
	(\tilde{M}'\otimes_{A}(c\varphi c'))(\tilde{m}'\otimes \tilde{p})   =   \tilde{m}'\otimes (c\varphi c')( \tilde{p}) = \tilde{m}'\otimes \varphi(\tilde{p}c)c'
	\]
	and
	$$
	\begin{array}{rcl}
	(c(\tilde{M}'\otimes_{A}\varphi)c')(\tilde{m}'\otimes \tilde{p}) & = & ((\tilde{M}'\otimes_{A}\varphi)((\tilde{m}'\otimes \tilde{p})c))c'\\
	& = & ((\tilde{M}'\otimes_{A}\varphi)(\tilde{m}'\otimes \tilde{p}c))c'\\
	& = & (\tilde{m}'\otimes \varphi(\tilde{p}c))c'\\
	& = & \tilde{m}'\otimes \varphi(\tilde{p}c)c'.
	\end{array}$$
	Henceforth $\tilde{M}'\otimes_{A}(c\varphi c')=c(\tilde{M}'\otimes_{A}\varphi)c'$, thus the morphism $\tilde{M}'\otimes_{A}-$ (see (\ref{eqn:morphism_2})) is a $(\C,\C)$-bimodule homomorphism.
	\par Second, we will prove that the morphism $\tilde{M}'\otimes_{A}-$
	 is a $G$-graded $(\C,\C)$-bimodule homomorphism, i.e. it is grade preserving. Consider $g\in G$ and $\varphi_g \in \Hom_A(\tilde{P},\tilde{Q})_g$. We want $\tilde{M}'\otimes_{A}\varphi_g\in \Hom_{A'}(\tilde{M}'\otimes_{A}\tilde{P},\tilde{M}'\otimes_{A}\tilde{Q})_g$, i.e.  by \cite[\S 1.2]{book:Marcus1999}, if for some $h\in G$ and $\tilde{m}'\otimes \tilde{p}\in (\tilde{M}'\otimes_{A}\tilde{P})_h$, then we must have $(\tilde{M}'\otimes_{A}\varphi_g)(\tilde{m}'\otimes \tilde{p})\in (\tilde{M}'\otimes_{A}\tilde{Q})_{hg}$. Beforehand, because $\tilde{m}'\otimes \tilde{p}\in (\tilde{M}'\otimes_{A}\tilde{P})_h$, by \cite[\S 1.6.4]{book:Marcus1999}, there exists some $x,y\in G$ with $h=xy$ such that $\tilde{m}'\in \tilde{M}'_{x}$ and $\tilde{p}\in \tilde{P}_{y}$. We have:
	$$
	\begin{array}{rcl}
		(\tilde{M}'\otimes_{A}\varphi_g)(\tilde{m}'\otimes \tilde{p}) & = &  \tilde{m}'\otimes \varphi_g(\tilde{p}) \in  \tilde{M}'_{x}\otimes_A \varphi_g(\tilde{P}_{y})\\
		& \subseteq & \tilde{M}'_{x}\otimes_A \tilde{Q}_{yg} \subseteq  (\tilde{M}'\otimes_A \tilde{Q})_{xyg}\\
		& = & (\tilde{M}'\otimes_A \tilde{Q})_{hg}.
	\end{array}
	$$
Henceforth, the morphism $\tilde{M}'\otimes_{A}-: \Hom_A(\tilde{P},\tilde{Q})\to \Hom_{A'}(\tilde{M}'\otimes_{A}\tilde{P},\tilde{M}'\otimes_{A}\tilde{Q})$ is a $\G$-graded $(\C,\C)$-bimodule homomorphism.
\end{proof}

 By Proposition \ref{prop:preequiv_example} and the observations made in Section \S\ref{s:graded_equivalence_data_over_C}, the following corollary is straightforward.

\begin{corollary}
	Let $P$ be a $\G$-invariant $\G$-graded $A$-module and $A'=\mathrm{End}_A(P)^{\mathrm{op}}$. If $\tensor[_{A}]{P}{}$ is a progenerator, then $P\otimes_{A'}-$ is a $\G$-graded Morita equivalence over $\C$ between $A'\text{-}\Gr$ and $A\text{-}\Gr$, with $P^{\ast}\otimes_{A}-$ as its inverse.
\end{corollary}

\begin{theorem}[Graded Morita II over $\C$]
Assume that $A$ and $A'$ are $\G$-graded Morita equivalent over $\C$ and let
$$
\xymatrix@C+=3cm{
A\text{-}\Gr \ar@<+.5ex>[r]^{\tilde{\mathcal{F}}} & A'\text{-}\Gr \ar@<+.5ex>[l]^{\tilde{\mathcal{G}}}
}
$$
be inverse $\G$-graded equivalences over $\C$. Then this equivalence is given by the following $\G$-graded bimodules over $\C$: $P=\tilde{\mathcal{F}}(A)$ and $Q=\tilde{\mathcal{G}}(A')$. More exactly, $P$ is a $\G$-graded $(A',A)$-bimodule over $\C$, $Q$ is a $\G$-graded $(A,A')$-bimodule over $\C$ and the following natural equivalences of functors hold:
\[\tilde{\mathcal{F}} \simeq P\otimes_{A}- \quad\text{ and }\quad \tilde{\mathcal{G}} \simeq Q\otimes_{A'}-.\]
\end{theorem}
\begin{proof}
By \cite[Corollary 10 (Gordon-Green)]{article:delRio1991}, we know that $P=\tilde{\mathcal{F}}(A)$ is a $\G$-graded $(A',A)$-bimodule, $Q=\tilde{\mathcal{G}}(A')$ is a $\G$-graded $(A,A')$-bimodule and that the following natural equivalences of functors hold: $\tilde{\mathcal{F}} \simeq P\otimes_{A}- $ and $\tilde{\mathcal{G}} \simeq Q\otimes_{A'}-.$ Moreover, we have that $\tensor[_{A}]{P}{}$ is a progenerator.
\par It remains to prove that $P$ and $Q$ are $\G$-graded bimodules over $\C$. We will only prove that $P$ is $\G$-graded bimodule over $\C$, because for $Q$ the reasoning is similar.
\par By the hypothesis, $A$ and $A'$ are $\G$-graded Morita equivalent over $\C$, hence $\tilde{\mathcal{F}}$ and $\tilde{\mathcal{G}}$ are over $\C$. Therefore the function:
$$
\xymatrix@C+=3cm{
\Hom_{A}(\tensor[_{A}]{A}{},\tensor[_{A}]{A}{}) \ar[r]^{\mathcal{\tilde{F}}}& \Hom_{A'}(\tensor[_{A'}]{P}{},\tensor[_{A'}]{P}{})
}$$
is an isomorphism of $\G$-graded $(\C,\C)$-bimodules, where if $f\in \Hom_{A}(\tensor[_{A}]{A}{},\tensor[_{A}]{A}{})$, we have that $(c_1 f c_2)(a)=f(ac_1)c_2$, for all $a\in A$, $c_1,c_2\in \C$. This means that the function
\[\alpha:A\to\End_{A'}(P)^{\op},\quad \alpha(a)=\F(\rho(a)),\,\text{for all }a\in A,\]
(where $\rho(c):a\mapsto ac$, for all $a\in A$) is an isomorphism of $\G$-graded $(\C,\C)$-bimodules. Moreover, by the bimodule structure definition of $P$ (see \cite{book:Anderson1992}), we have that $\alpha(a)(p)=pa$ for all $a\in A$ and for all $p\in P$.
\par It is clear that $\tensor[_{A}]{A}{}$ is $G$-invariant, hence by Lemma \ref{remark:same_stabilizer}, $P$ is also $G$-invariant. Henceforth, by Proposition \ref{prop:P_bimod_over_C}, $P$ is a $G$-graded $(A',\End_{A'}(P)^{\op})$-bimodule over $\C$. Consider the structural homomorphisms $\zeta: \C\to A$,  $\zeta': \C\to A'$ and $\zeta'': \mathcal{C}\to\End_{A'}(P)^{\op}$, thus for all $a\in A$ and for all $c_1,c_2\in\C$ we have:
\[\alpha(\zeta(c_1)a\zeta(c_2))=\zeta''(c_1)\alpha(a)\zeta''(c_2).\]
By taking $a=1_A$ and $c_2=1_{\C}$ we obtain $\alpha\circ \zeta=\zeta''$.
\par Let $g\in G$, $p_{g}\in P_{g}$ and $c\in \C$. We want $p_g\cdot c = \tensor[^{g}]{c}{}\cdot p_{g}$. We have:
\[
p_g\cdot c  =  p_g\cdot \alpha(\zeta(c)) =  p_g\cdot \zeta''(c) = \tensor[^{g}]{\zeta'(c)}{}p_g =  \tensor[^{g}]{c}{}p_g.
\]
Henceforth the statement is proved.
\end{proof}

\section{Conclusion} \label{s:Conclusion}
We have developed a $G$-graded Morita theory over a $G$-graded $G$-acted algebra for the case of finite groups.
\par In \coref{s:graded_bimodules_over_C} we recalled from \cite{MM2} the notions of a $G$-graded $G$-acted algebra, of a $G$-graded algebra over a $G$-graded $G$-acted algebra and that of a $G$-graded bimodule over a $G$-graded $G$-acted algebra and we gave some useful examples for each notion.
\par In \coref{s:graded_equivalence_data_over_C} we introduced the notion of a $G$-graded Morita context over a $G$-graded $G$-acted algebra and gave a standard example.
\par The main results are in \coref{s:Morita_theorems} where a notion of graded functors over $G$-graded $G$-acted algebras and of graded Morita equivalences over $G$-graded $G$-acted algebras are introduced and two Morita-type theorems are proved using these notions: 
We proved that by taking a $G$-graded bimodule over a $G$-graded $G$-acted algebra we obtain a $G$-graded Morita equivalence over the said $G$-graded $G$-acted algebra and that by being given a $G$-graded Morita equivalence over a $G$-graded $G$-acted algebra, we obtain a $G$-graded bimodule over the said $G$-graded $G$-acted algebra, which induces the given $G$-graded Morita equivalence.

\phantomsection

\end{document}